\documentclass[11pt]{article}
\usepackage{amsfonts, epsfig, amsmath, amssymb, color}
\usepackage{textcomp}
\usepackage{paralist}

\usepackage[english]{babel}

\renewcommand {\epsilon}{\varepsilon}

\newcommand{\CC}{\mathbb{C}}

\newcommand{\EE}{\mathbb{E}}

\newcommand{\PP}{\mathbb{P}}

\newcommand{\RR}{\mathbb{R}}

\newcommand{\bB}{\mathcal{B}}
\newcommand{\cC}{\mathcal{C}}

\newcommand{\fF}{\mathcal{F}}

\newcommand{\ffF}{\mathfrak{F}}

\newcommand{\fI}{\mathfrak{I}}

\newcommand{\fK}{\mathfrak{K}}

\newcommand{\fM}{\mathfrak{M}}

\newcommand{\e}{\varepsilon}
\newcommand{\la}{\lambda}

\newcommand{\si}{\sigma}

\newcommand{\lv}{\langle}
\newcommand{\rv}{\rangle}

\newcommand{\mto}{\mapsto}
\newcommand{\ra}{\rightarrow}

\newcommand{\lra}{\longrightarrow}
\newcommand{\ti}{\tilde}

\newcommand{\ind}{\mathbf{1}}

\newcommand{\lqq}{\leqslant}
\newcommand{\gqq}{\geqslant}

\newtheorem{thms}{Theorem}[section]

\newtheorem{thm}{Theorem}[section]

\newtheorem{prp}[thm]{Proposition}
\newtheorem{cor}[thm]{Corollary}

\newtheorem{lem}[thm]{Lemma}

\DeclareMathSymbol{\ophi}{\mathalpha}{letters}{"1E}

\renewcommand{\phi}{\varphi}

\newcommand{\be}{\begin{equation}}
\newcommand{\ee}{\end{equation}}
\newcommand{\ben}{\begin{equation*}}
\newcommand{\een}{\end{equation*}}

\newcommand{\ba}{\begin{equation}\begin{aligned}}
\newcommand{\ea}{\end{aligned}\end{equation}}

\DeclareMathOperator{\dist}{dist}
\DeclareMathOperator{\diag}{diag}

\newenvironment{proof}{\par\noindent{\bf Proof:}}{\hfill$\blacksquare$\par}

\newfont{\cyrfnt}{wncyr10}
\def\J3{\cyrfnt{\rm \u{\cyrfnt I}}}
\def\j3{\cyrfnt{\rm \u{\cyrfnt i}}}

\usepackage[]{color}
\definecolor{DarkGreen}{rgb}{0.1,0.7,0.3}   

\definecolor{DarkGreen}{rgb}{0.1,0.7,0.3}   

\allowdisplaybreaks[4]

\begin{document}
\title{
Averaging along foliated L\'evy diffusions
}



\author{Michael H\"ogele\footnote{Institut f\"ur Mathematik, Universit\"at Potsdam, Germany; hoegele@uni-potsdam.de}
\hspace{2cm}
Paulo Ruffino\footnote{IMECC, Universidade Estadual de Campinas, Brazil; ruffino@ime.unicamp.br} 
}

 \maketitle

\begin{abstract}
This article studies the dynamics of the 
strong solution of a SDE driven by a discontinuous L\'evy process 
taking values in a smooth foliated manifold with compact leaves.
It is assumed that it is \textit{foliated} in the sense that
its trajectories stay on the 
leaf of their initial value for all times a.s.. 
Under a generic ergodicity assumption for each leaf, 
we determine the effective behaviour of the system 
subject to a small smooth perturbation of order $\e>0$, 
which acts transversal to the leaves.  
The main result states that, on average,  
the transversal component of the perturbed SDE converges uniformly 
to the solution of a deterministic ODE as $\e$ tends to zero. 
This transversal ODE is generated by the average of the perturbing vector field 
with respect to the invariant measures of the unperturbed system and 
varies with the transversal height of the leaves. 
We give upper bounds for the rates of convergence 
and illustrate these results for the random rotations on the circle.  
This article 
complements the results by Gargate and Ruffino 
for SDEs of Stratonovich type to general L\'evy driven SDEs of Marcus type. 
\end{abstract}

\noindent \textbf{Keywords:} averaging principle; 
L\'evy diffusions on manifolds; foliated spaces; 
Marcus canonical equation\\

\noindent \textbf{2010 Mathematical Subject Classification: } 60H10, 60J60, 60G51, 58J65, 58J37.


\section{Introduction}
This article generalizes an averaging principle 
established for continuous semimartingales in Gonz\'ales and Ruffino \cite{GR13} to L\'evy diffusions 
containing a jump component. 
The system under consideration is the strong solution of a stochastic differential equation (SDE)
driven by a discontinuous L\'evy noise with values in a smooth Riemannian manifold $M$ 
equipped with a foliation structure $\fM$. 
That means there exists an equivalence relation on $M$, which defines a family of 
immersed submanifolds of constant dimension~$n$, which cover $M$. 
The elements of $\fM$ are called the leaves of the foliation. 
Moreover the solution flow of the SDE is assumed to respect the foliated structure $\fM$ 
in the sense that each of the discontinuous solution paths of the SDE stays 
on the corresponding leaf of its initial condition for all times almost surely. 
We further assume the existence of a unique invariant measure for the SDE on each leaf. 

If this system is perturbed by a smooth deterministic vector field $\e K$ 
transversal to the leaves with intensity $\e>0$, 
the foliated structure of the solution is destroyed due to the appearance 
of a (smooth) transversal component in the trajectories. 
We study the effective behaviour of this transversal component in the limit as $\e$ tends to $0$. 

The main idea is the following. 
Consider the solution along the rescaled time $t/\e$, 
its foliated component approximates the ergodic average behaviour for small $\e$. 
Hence the essential transversal behaviour is captured by an ODE for the transversal component 
driven by the  vector field $K$ instead of $\e K$, which is averaged by the ergodic invariant measure on the leaves.   
Note that the intensity of the original perturbation $\e K$ cancels out by the time scaling $t/\e$. 
This is the result of Theorem~\ref{thms: main result} and will be referred to as an averaging principle. 
Our calculations here also determine upper bounds for the rates of convergence and a probabilistic robustness result. 

The heuristics of an averaging principle consists in replacing 
the fine dynamical impact of a so-called fast variable on the dynamics 
of a so-called slow variable by its averaged statistical static influence.  
For references on the vast also classical literature on averaging for 
deterministic systems see e.g. the books by V. Arnold \cite{V-Arnold} and 
Saunders, Verhulst and Murdock \cite{SVM} and the numerous citations therein. 
For stochastic systems among many others we mention the book by 
Kabanov and Pergamenshchikov \cite{Kabanov-Pergamenshchikov} and the references therein 
which gives an excellent overview on the subject. See also \cite{Borodin-Freidlin, Khasminski-krylov}.
An inspiration for this article also goes back to the work  \cite{Li} by Li, 
where she established an averaging principle for the particular case of 
completely integrable (continuous) stochastic Hamiltonian systems. 
In Gonz\'ales and Ruffino \cite{GR13} these results have been generalized 
to averaging principles for perturbations of Gaussian diffusions on foliated spaces. 
This article completes this result for general L\'evy driven foliated diffusions. 

The article is organized as follows. In the next Section we present 
the main result and an example to illustrate the 
averaging phenomenon. Section 3 is dedicated to 
the fundamental technical Proposition \ref{lem: preliminary}, 
where the stochastic Marcus integral technique is applied and whose estimates turn out to be the 
basis for the rates of convergence of the main theorem. Section 4 deals with 
the averaging on the leaves. 
In Section 5 we prove the main theorem. Further details of the calculations in 
the example of Section 2.3 are provided in an Appendix.

\section{Object of study and main results} 

\subsection{The set up } 
Let $M$ be a connected smooth Riemannian manifold with an $n$-dimensional 
smooth foliation $\fM$. Given an initial 
condition $x_0\in M$, we assume that there exists a bounded neighbourhood 
$U\subset M$ of the corresponding compact leaf 
$L_{x_0}$ such that there exists a diffeomorphism $\varphi: U 
\rightarrow L_{x_0}\times V$, where $V\subset \RR^d$ is a connected open set 
containing the origin. The neighbourhood $U$ is taken small enough 
such that the derivatives of $\varphi$ are 
bounded (say, $U$ is precompact in $M$). The second coordinate, called the 
vertical coordinates of a point 
$q \in U$ will be denoted by the  projection $\Pi: U \rightarrow V$ with 
$\Pi(q)\in V$, i.e. 
$\varphi(q)=(u, \Pi(q))$ for some $u\in L_{x_0}$.  
Hence for any fixed $v\in V$, the inverse image $\Pi^{-1}(v)$ is the compact 
leaf $L_x$, where $x$ is any point in $U$ such that the 
vertical projection satisfies $\Pi(x)=v$. The 
components of the vertical projection are denoted by
\[
 \Pi(q) = \Big(\Pi_1(q), \ldots , \Pi_d(q) \Big) \in V \subset \RR^d,
\]
for any $q\in U$. 

We are interested in a L\'evy driven SDE with values in $M$ embedded in an 
Euclidean space, 
whose solutions respect the foliation. 
Since such a solution necessarily satisfies a canonical Marcus equation, 
also known as generalized Stratonovich equation in the sense of Kurtz, Pardoux and Protter \cite{KPP95},  
we consider the stochastic differential equation 
\begin{equation}\label{eq: SDE}
d X_t = F_0(X_t) dt + F(X_t) \diamond d Z_t, \qquad X_0 = x_0,
\end{equation}
which consists of the following components. 
\begin{enumerate}
\item Let $F \in \cC^1(M; L(\RR^{r}; T\fM))$, such that the map $x\mapsto F(x)$ is $\cC^1$ 
such that for each $x\in M$ the linear map $F(x)$ sends a vector $z\in \RR^r \mto F(x) z \in T_x L_x$ 
to the respective tangent space of the leaf. 
Furthermore we assume that $F$ and $(DF)F$ are globally Lipschitz continuous on $M$ 
with common Lipschitz constant $\ell>0$. 
We write $F_i(x)= F(x) e_i$, for $e_i$ the canonical basis of $\RR^r$. 
 \item Here $Z = (Z_t)_{t\gqq 0}$ with $Z_t = (Z^1_t, \dots, Z^r_t)$ is a L\'evy process in $\RR^r$ with respect to 
a given filtered probability space $(\Omega, \fF, (\fF_t)_{t\gqq 0}, \PP)$ with characteristic triplet $(0, \nu, 0)$. 
It is assumed that the filtration satisfies the usual conditions in the sense of Protter \cite{Pr04}.
It is a consequence of the L\'evy-It\^o decomposition of~$Z$ 
that $Z$ is a pure jump process with respect to a L\'evy measure $\nu: \bB(\RR^r) \ra [0, \infty]$ 
which satisfies $\nu\{0\} = 0$ and permits the 
existence of an exponential moment of order $\kappa > \ell$ given by 
\begin{equation}\label{eq: exponential moment}
\int_{\RR^r} \big(e^{\kappa \|y\|}\wedge \|y\|^2\big)\; \nu(dy)< \infty.
\end{equation}
where $\ell$ the Lipschitz constant of the function $x \mapsto F(x)$ introduced in point 1. 
For details we refer to the monographs of Sato \cite{Sa99} or Applebaum \cite{Ap09}.
 \item Equation (\ref{eq: SDE}) is defined as 
\begin{equation}\label{eq: SDE unperturbed integral form}
X_t = x_0 + \int_0^t F_0(X_s) ds + \int_0^t F(X_{s-}) d Z_s + \sum_{0 < s\lqq t} (\Phi^{F \Delta_s Z}(X_{s-})-X_{s-}- F(X_{s-}) \Delta_s Z),
\end{equation}
where the function $\Phi^{Fz}(x) = Y(1, x ; Fz)$ and $Y(t, x; Fz)$ stands for 
the solution of the ordinary differential equation 
\begin{equation}\label{eq: increment ode}
\frac{d}{d\si } Y(\si) = F(Y(\si)) z,  \qquad \mbox{ with initial condition } 
Y(0) = x \in M, \quad z\in \RR^r. 
\end{equation}
\end{enumerate}
This article studies the situation where a foliated SDE is perturbed by a
transversal smooth vector field $\e K$ with $\e>0$ in the limit for $\e \searrow 0$. 
\begin{enumerate}
 \item[4.] For $K: M \ra TM$ a smooth, globally Lipschitz continuous vector field 
we denote by $X^\e$, $\e>0$ the solution of the perturbed system 
\begin{equation}\label{eq: SDE perturbed}
d X^\e_t = F_0(X^\e_t) dt +  F(X_t) \diamond d Z_t + \e K(X^\e_t) dt, \qquad X^\e_0 = x_0.
\end{equation}
\end{enumerate}

\begin{thms}[\cite{KPP95}, Theorem 3.2 and 5.1]
\begin{enumerate}
\item Under the previous assumptions notably item \mbox{1.- 3.,} there is a unique semimartingale $X$ 
which is a strong global solution of (\ref{eq: SDE}) 
in the sense of equation (\ref{eq: SDE unperturbed integral form}). It has a c\`adl\`ag version and 
is a (strong) Markov process. 
\item Under the previous assumptions in particular item \mbox{1.-4.,} there is a unique semimartingale $X^\e$ 
which is a strong global solution of equation (\ref{eq: SDE perturbed}) in the 
sense 
of equation (\ref{eq: SDE unperturbed integral form}), where $F_0$ is replaced by $F_0 +\e K$. 
The perturbed solution $X^\e$ has c\`adl\`ag paths almost surely and is a (strong) Markov process. 
\end{enumerate}
\end{thms}

With the previously mentioned embedding results in mind 
we are now in the position to apply Proposition 4.3 in Kurtz, Pardoux and Protter \cite{KPP95}, 
which states the following support theorem. Under the aforementioned conditions 
$\PP(X_0 \in M) =1$, it  implies that $\PP(X_t \in M, \quad \forall t\gqq 0) = 
1$. 
The same result applied again on the leaves yields in particular 
that each solution is \textit{foliated} in the sense 
that it stays on the leaf of its initial condition, i.e. 
$\PP(X_0 \in L_{x_0}) =1$ implies that $\PP(X_t \in L_{x_0},\quad \forall t\gqq 
0) = 1$. 
We shall call an SDE of the type (\ref{eq: SDE}) which admits a foliated solution 
a \textit{foliated stochastic differential equation}.

\subsection{The main results} 

We assume that each leaf $L_{q}\in \fM$ passing through $q \in M$ 
has associated a unique invariant measure $\mu_{q}$ of the unperturbed 
foliated 
system (\ref{eq: SDE}) with initial condition $x_0 = q$. 

Let $\Psi: M\ra \RR$ be a differentiable function. 
We define the average of $\Psi$ with respect to $\mu_q$ in the following way. 
If $v$ is the vertical coordinate of $q$, that is $\phi(q) = (u,v)$, we 
define 
\begin{equation}\label{def: average}
Q^\Psi(v) := \int_{L_q} \Psi(y) \mu_q(dy).  
\end{equation}
With respect to the coordinates given by $\phi$, the perturbing vector field 
$K$ is written as 
\[ 
 d\Pi(K)= \big(d\Pi_1 (K), \ldots, d\Pi_d(K)).
\]

\begin{enumerate}
 \item[\textbf{Hypothesis 1:} ] For $i=1, \dots, d$ the function 
\begin{equation}\label{eq: average}
v \mapsto Q^{d\Pi_i K}(v)
\end{equation}
is globally Lipschitz continuous. 
\end{enumerate}
This ensures in particular that for each $w\in \RR^d$ the ordinary differential equation 
\begin{equation}\label{def: w}
\frac{d w}{dt}(t) = (Q^{d\Pi_1( K)}, \dots,  Q^{d\Pi_d( K)})\left(w (t)\right), \qquad w(0) = w
\end{equation}
has a unique global solution. 

We are going to consider the average on the leaves of each real valued 
functions $d \Pi_i(K)$, with $i=1, \ldots , d$. In general, there is no 
standard rate of convergence in the ergodic 
theorem, see e.g. Krengel \cite{Krengel}, Kakutani and Petersen 
\cite{Kakutani-Petersen}. We are going to consider a prescribed rate of 
converge in time average, following the classical 
approach as in Freidlin and Wentzell 
\cite[Chap. 7.9]{Freidlin-Wentzell}, where they deal with an averaging 
principle with convergence in probability. 

\begin{enumerate}
 \item[\textbf{Hypothesis 2: }] For any $x_0\in M$ and $p\gqq 2$ there exists a 
positive, bounded, decreasing function $\eta: [0,\infty) \ra [0,\infty)$ 
 which estimates the rate of 
convergence of the ergodic unperturbed 
dynamic in  each leaf  $L_{x_0}$ in the sense that for all $i \in \{1, \dots, d\}$
\begin{equation} \label{def: function eta}
\left( \mathbb{E} \left| \frac{1}{t}\int_0^t d\Pi_i K (X_s(x_0) )\, ds - 
Q^{d\Pi_i (K)} (\Pi (x_0)) \right|^p \right)^{\frac{1}{p}} \lqq \eta(t), \qquad \mbox{ for all }t\gqq 0.
\end{equation}
\end{enumerate}

%

For $\e>0$ and $x_0\in M$ let $\tau^\e$ be the first exit time of 
the solution $X^\e(x_0)$ of equation (\ref{eq: SDE perturbed}) 
from the aforementioned foliated coordinate neighborhood $U$. 

\begin{thms}\label{thms: main result}
Assume that the unperturbed foliated system (\ref{eq: SDE}) on $M$ satisfies Hypothesis 1 and~2. 
Let $w$ be the solution of the deterministic ODE in the transversal component $V\subset \RR^n$. 
\begin{equation}\label{def: w}
\frac{d w}{dt}(t) = (Q^{d\Pi_1( K)}, \dots,  Q^{d\Pi_d( K)})\left(w (t)\right),
\end{equation}
with initial condition $w(0) = \Pi(x_0) = 0$. Let $T_0$ be the time that $w(t)$ reaches the boundary of $V$. 

Then we have that:
\begin{enumerate}
 \item For all $0< t< T_0$, $\beta \in (0, \frac{1}{2})$, $p\gqq 2$ and $\la < 1$ 
\begin{align*}
\left(\EE\left[\sup_{s\lqq t} |\Pi\big(X^\e_{\frac{s}{\e}\wedge \tau^{\e}}\big) 
- w(s)|^p\right]\right)^{\frac{1}{p}} 
\lqq t \left[ \e^{\la} h(t, \e) +  \ \eta \left( t |\ln \epsilon |^{\frac{2\beta}{p}} \right)\, \exp{\{Ct\}} \right], 
\end{align*}
where $h(t,\e)$ is continuous and converges to zero, when $\e$ or $t$ do so and 
$C$ is a positive constant.
 \item For $\gamma > 0$, let 
\[
T_\gamma := \inf\{t>0~|~\dist(w(t), \partial V) \lqq \gamma\}. 
\]
The exit times of the two systems satisfy the estimates 
\[
\PP(\e \tau^{\e} < T_\gamma) \lqq \gamma^{-p}T_\gamma^p \left[ \e^{\la} 
h(T_\gamma, \e) +  \ \eta 
\left( T_\gamma |\ln \epsilon 
|^{-\frac{2\beta}{p}} \right) e^{CT_\gamma} \right]^p.  
\]
\end{enumerate}
\end{thms}
The second part of the theorem above guarantees the robustness of the averaging phenomenon 
in transversal direction. \\

\subsection{Example. Perturbed random rotations: The Gamma process on the 
circle} 

As a simple but illustrative example of the phenomenon, 
 consider  $M= \RR^3 \setminus \{ (0,0,z), z\in \RR \}$
with the 1-dimension horizontal circular foliation of $M$ where the leaf
passing through a point $q=(x,y,z)$ is given by the (non-degenerate) horizontal 
circle at height $z$:
\[
L_q = \{ (
\sqrt{x^2+ y^2}
\cos \theta, \sqrt{x^2+ y^2} \sin 
\theta, z),\quad \theta \in [0,2 \pi) \}.
\]
This foliation (cf. \cite{GR13}) is an example where the 
transversal space 
is richer than the leaves themselves, hence the range of the impact for different 
perturbations is relatively large. 

Let the process $Z = (Z_t)_{t\gqq 0}$ be a Gamma process in $\RR$ with 
characteristic triplet $(0, \nu, 0)$, where $\nu(dy) = \frac{e^{-\theta |y|}}{|y|}$ 
is the corresponding L\'evy jump measure with a rate parameter $\theta>0$. 
It satisfies the integrability condition $\int_{\RR} (e^{\kappa |y|} \wedge |z|^2) \nu(dz) < \infty$ 
required in (\ref{eq: exponential moment}). 
The L\'evy-It\^o decomposition of $Z$ yields the almost sure decomposition 
\begin{equation}\label{eq: Levy-Ito}
Z_t = \int_0^t \int_{|y|\lqq 1} y \ti N(ds, dy) + \int_{|y|> 1} y N(ds, dy),
\end{equation}
where $N$ is the random Poisson measures with intensity measure $dt \otimes 
\nu$.
$\ti N$ denotes its compensated counterpart. 
Consider the foliated linear SDE on
$M$ consisting of random rotations:
\begin{equation}\label{eq: SDE example}
 dX_t = \Lambda X_t ~\diamond dZ_t, \qquad X_0 = q_0 = (x_0, y_0, z_0),
\end{equation}
where 
\[
\Lambda 
= \left( 
\begin{array}{ccc}
0 & -1 & 0  \\
1 & 0 & 0 \\
0 & 0 & 0
\end{array} 
\right). 
\]
Equation (\ref{eq: SDE example}) is defined as follows: 
First note that $\Lambda^2 =\diag(-1,-1,0)$. Secondly, note that for $z\in \RR$, $z\neq 0$ the solution flow $\Phi$ of the equation 
\[
\frac{d}{d\si }Y(\si) = F(Y(\si)) z, \qquad Y(0) = q \qquad \mbox{ where }F(\bar q) = \Lambda \bar q
\]
is obtained by a simple calculation as
\[
\Phi^{F z}(q) = Y(1; q) 
= \left(\begin{array}{c}  x \cos(z) - y \sin(z) \\ x \sin(z) + y \cos(z) \\ z\end{array}\right),
\]
such that 
\begin{align*}
X_t &= q_0 + \int_0^t \Lambda X_{s-} z \ti N(ds, dz) +  \sum_{0 < s\lqq t} (\Phi^{F \Delta_s Z}(X_{s-}) - X_{s-}- F(X_{s-}) \Delta_s Z).
\end{align*}

We ignore for the moment the constant third component $X_t^3 = z_0$ for all $t\gqq 0$ almost surely 
and write $\bar X = (X^1, X^2)$ for convenience. 
Using the chain rule of the Marcus integral, 
as stated in Proposition 4.2 in \cite{KPP95}, 
we verify for $\chi(x,y) := x^2 + y^2$ we get 
\begin{align}\label{ex: non radial part}
d \chi(\bar X_t) = -2 \bar X_{t-} \Lambda \bar X_{t-} \diamond d Z_t = 0.
\end{align}
In fact, $\bar X^\e$ can be defined equally as the projection of $Z$ on the unit circle. 
If we identify the plane, where $\bar X$ takes its values, with the complex 
plane $\CC$, 
one verifies easily that
\[
\bar X_ t= e^{i Z_t}.
\]
In fact we obtain the well-known L\'evy-Chinchine representation of the characteristic function 
for any $p\in \RR$:
\begin{align*}
\EE[\bar X_t^p ]&= \EE[e^{ipZ_t}] = \exp(t \Psi(p)),  \\
\Psi(p) &= \int_{\RR^d} (e^{ipy}-1-iyp\ind\{|y|\lqq 1\}) d\nu(y).
\end{align*}
It is a direct consequence of Lemma \ref{lem: ergodicity Gamma process} that for any uniformly distributed 
random variable on a circle in the complex plane centered in the origin we have 
$Z_t \stackrel{d}{\lra} U$ as $t \ra \infty$. 
The invariant measures $\mu_q$ in the leaves $L_q$ passing through  points $q\in M$
are therefore given by normalized Lebesgue measures in the circle $L_q$ centered in 
$0$ with radius $|(x,y)|$. 
We investigate the effective behaviour of a small
transversal perturbation of order $\epsilon$:
\[
 dX^{\epsilon}_t =  
\Lambda X^{\epsilon}_t \;\diamond dZ_t  +
 \epsilon K (X^{\epsilon}_t)\ dt
\]
with initial condition $q_0=(1,0,0)$. 
In this example we shall consider two classes of perturbing vector field $K$. 

\bigskip

\noindent (A) Constant perturbation $\e K$. Assume that  the perturbation is
given by a constant vector field $K=(K_1, K_2, K_3)\in \RR^3$ 
with respect to Euclidean coordinates in $M$.  
Then, the horizontal average of the radial component 
\[Q^{d\Pi_r K}(r_0, z_0) = \int_{L_q} \lv (K_1, K_2)^T, y\rv d\mu_q(y) = 0, \qquad q = (\theta_0, r_0, z_0)\]
and the vertical $z$-component is constant  $Q^{d\Pi_z K }=K_3$.  
We verify (\ref{def: function eta}) and obtain with the help of (\ref{ex: non radial part}) 
the rate function $\eta\equiv 0$. In fact trivially: 
\[
\Big(\EE[|\frac{1}{t}\int_0^t (d \Pi_r K)(X_s) ds - Q^{d\Pi_r K}(r_0, 
z_0)|^p]\Big)^\frac{1}{p} = 0,
\]
and
\[
\Big(\EE[|\frac{1}{t}\int_0^t (d \Pi_z K)(X_s) ds - Q^{d\Pi_z K}(r_0, 
z_0)|^p]\Big)^\frac{1}{p} 
= K_3 - K_3 = 0. 
\]
Hence the transversal component in Theorem \ref{thms: main result} for initial condition $q_0 = (1, 0, 0)$ 
is given by $w(t)=(1, K_3t)$ for all $t\gqq 0$.
Theorem \ref{thms: main result} establishes a minimum rate of
convergence to zero of the difference between each of the transversal
components. Hence for the radial component of the perturbed systems $w_1(t)\equiv 1$ and
$\Pi_r (X^{\epsilon}_{\frac{t}{\epsilon}\wedge \tau^{\epsilon}})$ it holds that, 
for $p\gqq 2$ and $\gamma \in (0,1)$
\[
\left[\mathbb{E}\left( \sup_{s\lqq t} 
\left| \Pi_r (X^{\epsilon}_{\frac{t}{\epsilon}\wedge \tau^{\epsilon}}) - 1 
\right| ^p\right)\right]^{\frac{1}{p}}\lqq \e^\gamma t.
\]
For the second transversal component, we have that 
\[
|  \Pi_z
\left( X^{\epsilon}_{\frac{t}{\epsilon}\wedge \tau^{\epsilon}}
\right) - w_2(t)| \equiv 0
\]
for all $t\gqq 0$ and the convergence of the theorem is trivially verified.

\bigskip 

\noindent (B) Linear perturbation $\e K(x,y,z) =  \e (x,0,0)$.
For the sake of simplicity, we consider a one dimensional and horizontal linear
perturbation, which in this case can be written in the
form $K(x,y,z)= (x, 0,0 )$. 
The $z$-coordinate average vanishes trivially. 
For the radial component, we have that  $d \Pi_r K(q) = r \cos^2(\theta)$,
where $\theta$ is the angular coordinate of $q = (\theta, r, z)$ whose distance to the $z$-axis
(radial coordinate) is $r$. Hence the average with respect to the invariant
measure on the leaves is given by  
\[ 
Q^{d\Pi_r K}(\theta, r, z) = \frac{1}{2\pi} \int_0^{2 \pi} r \cos^2(\theta) d\theta = \frac{r}{2}
\]
for leaves $L_q$ with radius~$r$. We verify the convergence (\ref{def: function eta}) of Hypothesis 2 
for the radial component 
and $p=2$. Elementary calculations, which can be found in the Appendix after Lemma \ref{lem: ergodicity Gamma process}, 
show that
\begin{align*}
&\EE\Big[\Big|\frac{1}{t}\int_0^t d \Pi_r K(Z_s, 0, 0) ds - Q^{d\Pi_r K}(q_0)\Big|^2\Big]\\
&= \EE\Big[\Big|\frac{1}{t}\int_0^t r \cos^2(Z_s) ds - \frac{r}{2}\Big|^2\Big]\\
&= \Big(\frac{r}{t}\Big)^2 2\int_{0}^t\int_0^\si \EE\Big[\cos^2(Z_s)\cos^2(Z_\si)\Big] dsd\si
 + \frac{r^2}{t} \int_0^t \frac{1}{2} \exp(- Cs) ds - \frac{r^2}{4},
\end{align*}
where the first term can be estimated by constants $a, b>0$ 
\begin{align*}
\Big(\frac{r}{t}\Big)^2 2\int_{0}^t\int_0^\si \EE\Big[\cos^2(Z_s)\cos^2(Z_\si)\Big] dsd\si
&\lqq \Big(\frac{r}{t}\Big)^2 \frac{1}{4} (a + bt + t^2) \stackrel{t\nearrow \infty}{\lra} \frac{r^2}{4}.
\end{align*}
Combining the previous two results and 
taking the square root, the rate of convergence is of order $\eta(t) = 
C/\sqrt{t}$ for a positive consant $C$ as $t\nearrow\infty$.

For an initial value $q_0 = (x_0, y_0, z_0) = (r_0 \cos(u_0), r_0\sin(u_0), z_0)$ 
the transversal system stated in Theorem \ref{thms: main result} is then 
$w(t)= (e^\frac{t}{2} r_0,z_0)$. Hence the result guarantees that the radial part 
$\Pi_r \big( X^{\epsilon}_{\frac{t}{\epsilon}\wedge \tau^{\epsilon}} \big)$
must have a behaviour close to the exponential $e^{\frac{t}{2}}$ in the sense that 
\[
\left[\mathbb{E}\left( \sup_{s\lqq t} 
\left| \Pi_r  \big( X^{\epsilon}_{\frac{t}{\epsilon}\wedge \tau^{\epsilon}} 
\big) - e^{\frac{t}{2}}\right|
^2\right)\right]^{\frac{1}{2}}
\lqq  C t  \e^{\la}  +  C \sqrt{t} \exp(Ct) |\ln \epsilon 
|^{-\frac{\beta}{p}}, 
\]
tends to zero when $\epsilon\searrow 0$.

\section{Transversal perturbations} 

Next proposition gives information on the order of which the perturbed 
trajectories approach the unperturbed trajectories when $\e$ goes to zero.

\begin{prp}\label{lem: preliminary} 
 For any Lipschitz test function $\Psi: M \ra \RR$ 
there exist positive constants $k_1, k_2$ such that for all $T\gqq 0$ we have 
\begin{equation}\label{eq: compact bounded}
\left(\EE\left[\sup_{t\lqq T\wedge \tau^\e}|\Psi(X^\e_t(x_0) ) - \Psi(X_t(x_0))|^p\right]\right)^{\frac{1}{p}} 
\lqq k_1\, \e\, T \exp\left(k_2\; T^p\right). 
\end{equation}
The constants $k_1$ and $ k_2$ depend on the upper bounds of the norms for the 
perturbing vector field $K$ in $U$, 
on the Lipschitz coefficient of $\Psi$ and on the derivatives of the vector 
fields $F_0, F_1 \dots, F_r$ with respect to the coordinate system. 

\end{prp}

\begin{proof} 
First we rewrite $X^\e$ and $X$, the solutions of equation (\ref{eq: SDE}) and (\ref{eq: SDE perturbed}), 
in terms of the coordinates given by the diffeomorphism $\phi$:
\begin{align*}
(u_t, v_t) := \phi(X_t) \qquad &\mbox{ and }\qquad (u^\e_t, v^\e_t) := 
\phi(X^\e_t).
\end{align*}
Exploiting the regularities of $\Psi$ and $\phi$ we obtain 
\begin{align}\label{eq: ungleichung 1}
|\Psi(X^\e_t)- \Psi(X_t)| &= |\Psi \circ \phi^{-1}(u^\e_t, v^\e_t) -\Psi \circ \phi^{-1}(u_t, v_t) | \nonumber \\
&\lqq C_0 |(u^\e_t-u_t, v^\e_t- v_t)|\lqq C_0 (|u^\e_t-u_t|+| v^\e_t- v_t|).
\end{align}
for $C_0 := Lip(\psi) \sup_{y \in  U} |\phi^{-1}(y)|$. The proof of the 
statement consists in calculating estimates for each summand on the right hand 
side of equation above.
We define 
\begin{align*}
\ffF_i &:= (D\phi) \circ F_i\circ \phi^{-1} \qquad \mbox{ for }i =0, \dots, n,\\
\fK &:= (D\phi) \circ K \circ \phi^{-1}, 
\end{align*}
which all together with their derivatives are bounded. 
Considering the components in the image of $\phi$ we have: 
\[
\fK = (\fK_H, \fK_V),
\]
with $ \fK_H \in TL_{x_0}$  and  
$\fK_V \in TV \simeq \RR^d$. The chain rule proved in Theorem 4.2  of \cite{KPP95} 
yields for equation (\ref{eq: SDE perturbed}) the following form in $\phi$ 
coordinates, written componentwise as 
\begin{align}
d u_t^{\e} &=  \ffF_0(u_t^{\e}, v_t^{\e}) dt +  \ffF(u_t^\e, 
v_t^\e) \diamond d Z_t + \e\, \fK_H(u_t^\e, v_t^\e) dt &\mbox{ with } 
u_t^{\e}\in L_{x_0}, 
\label{eq: foliations SDE in koordinaten}\\
d v_t^{\e} &= \e\, \fK_V(u_t^\e, v_t^\e) dt &\mbox{ with } v_t^{\e}\in V.\ \ 
\label{eq: transversale SDE in koordinaten}
\end{align}
We start with estimates on the transversal components $|v^\e-v|$. 
Let $|K|$ and hence 
$|\fK|$ be bounded by a universal constant, $C_1>0$, say. 
Therefore equation (\ref{eq: transversale SDE in koordinaten}) yields the estimate 
\begin{equation}\label{eq: Gv estimate}
\sup_{t\lqq T \wedge \tau^\e} |v_t^\e - v_t| 
\lqq \e \sup_{t\lqq T \wedge \tau^\e} \int_0^t ~|\fK_V(u_s^\e, v_s^\e)| ~ds 
\lqq C_1 \e T.
\end{equation}
We continue with estimates on the difference of the `horizontal' component. 
Recall that for $t< \tau^\e$ formally 
\begin{align*}
u^{\e}_t - u_t & =  \int_0^t (\ffF_0(u_s^{\e}, v_s^{\e}) - \ffF_0(u_s, v_s)) ds 
+ \int_0^t (\ffF(u_s^\e, v_s^\e) -\ffF(u_s, v_s)) \diamond d Z_s + \e\, \int_0^t \fK_H(u_s^\e, v_s^\e) ds.
\end{align*}
This equality is defined as 
\begin{align*}
u_t^{\e} - u_t & = \int_0^t [\ffF_0(u^\e_s, v^\e_s) - \ffF_0(u_s, v_s)] ds \\
&\quad + \int_0^t [\ffF(u_{s-}^\e, v^\e_{s-}) - \ffF(u_{s-}, v_{s-}) ]d Z_s\\
&\quad + \sum_{0 < s\lqq t} 
\big[(\Phi^{\ffF \Delta_s Z}(u^\e_{s-}, v^\e_{s-}) -\Phi^{\ffF \Delta_s Z}(u_{s-}, v_{s-}))\\
&\qquad\qquad -(u^{\e}_{s-} -u_{s-}) - (\ffF(u^\e_{s-}, v^\e_{s-})-\ffF(u_{s-}, v_{s-})) \Delta_s Z\big] \\
&\quad+ \e \int_0^t \fK_H(u_s^\e, v_s^\e) ds.
\end{align*}
Since $p\gqq 2$ this leads to 
\begin{align}
 |u_t^{\e} - u_t|^p  & \lqq 4^{p-1} \Big|\int_0^t  \ffF_0(u^\e_s, v^\e_s) - \ffF_0(u_s, v_s) ds \Big|^p + ~4^{p-1} C_1^p \e^p t^p\nonumber \\
 & \qquad + 4^{p-1} \Big|\int_0^t [\ffF(u_{s-}^\e, v^\e_{s-}) - \ffF(u_{s-}^\e, v^\e_{s-}) ]d Z_s\Big|\nonumber\\
 & \qquad +  4^{p-1} \Big|\sum_{0< s\lqq t} \Phi^{\ffF \Delta_s Z}(u^\e_{s-}, v^\e_{s-}) 
-\Phi^{\ffF \Delta_s Z}(u_{s-}, v_{s-}) -(u^{\e, i}_{s-} -u^{i}_{s-}) \nonumber\\
 & \qquad \qquad - (\ffF(u^\e_{s-}, v^\e_{s-})-\ffF(u_{s-}, v_{s-})) \Delta_s Z\Big|^p. \label{eq: pivot}
\end{align}
We now estimate the terms of the right-hand side in (\ref{eq: pivot}). 
The first term on the right-hand side is dominated by Jensen's inequality 
and equation (\ref{eq: Gv estimate}) 
\begin{align*}
\Big|\int_0^t  \ffF_0(u^\e_s, v^\e_s) - \ffF_0(u_s, v_s)  ds \Big|^p 
& \lqq  \Big( \int_0^t C_2 |(u^\e_s-u_s, v^\e_s-v_s)|ds \Big)^p\\
&\qquad \lqq C_2^p  \Big(\int_0^t (|u_s^\e-u_s| + |v_s^\e - v_s|) ds \Big)^p\\
&\qquad \lqq C_2^p (2t)^{p-1} \left(\int_0^t |u_s^\e-u_s|^p ds + \int_0^t |v_s^\e-v_s|^p ds\right)\\
&\qquad \lqq C_2^p (2t)^{p-1} \left(\int_0^t |u_s^\e-u_s|^p ds + C_1^p \e^p t^{p+1}\right)\\
&\qquad \lqq C_2^p (2t)^{p-1} \int_0^t |u_s^\e-u_s|^p ds + (2 C_1 C_2)^{p} t^{2p} \e^p.
\end{align*}
The term in the second line has the following representation with respect to the random Poisson measure associated to $Z$
\begin{align*}
\int_0^t [\ffF(u_{s-}^\e, v^\e_{s-}) - \ffF(u_{s-}, v_{s-}) ]d Z_s 
&= \int_0^t \int_{\RR^r\setminus\{0\}} [\ffF(u_{s-}^\e, v^\e_{s-}) - \ffF(u_{s-}, v_{s-}) ] z \ti N(ds, dz).
\end{align*}
By Kunita's first inequality for the supremum of integrals with respect to the 
compensated random Poisson measure integrals, 
as stated for instance in Theorem 4.4.23 in \cite{Ap09}, and inequality (\ref{eq: Gv estimate}) we obtain 
\begin{align}
&\EE\left[\sup_{t\in [0, T]}\Big|\int_0^t \int_{\RR^r} [\ffF(u_{s-}^\e, v^\e_{s-}) - \ffF(u_{s-}, v_{s-}) ] z \ti N(ds, dz)\Big|^p \right] \nonumber\\
& \lqq C_3 \; \left(\int_{\RR^r}\|z\|^2\nu(dz) \right)^{p/2}\;
\EE\left[\left(\int_0^T |\ffF(u_{s-}^\e, v^\e_{s-}) - \ffF(u_{s-}, v_{s-})|^2 ds\right)^{p/2} \
\right]\nonumber\\
& \qquad + C_3 \;\int_{\RR^r} \|z\|^p \nu(dz)\;\EE\left[\left(\int_0^T |\ffF(u_{s-}^\e, v^\e_{s-}) - \ffF(u_{s-}, v_{s-})|^p ds\right) \right]\nonumber\\
& \lqq (2 C_2)^{p}C_3 (C_4 T^{\frac{p}{2}-1}+ C_5 ) \int_0^T \EE\left[\sup_{s\in [0, t]} |u^\e_s - u_s|^p \right] ds
 + (2 C_1 C_2)^{p} C_3(C_4+ C_5) T^{p+1} \e^p\\
& \lqq (2 C_2)^{p}C_3 (C_4 + C_5 )(T^{\frac{p}{2}-1}+1) \int_0^T \EE\left[\sup_{s\in [0, t]} |u^\e_s - u_s|^p \right] ds
 + (2 C_1 C_2)^{p} C_3(C_4+ C_5) (T^{\frac{3p}{2}} + T^{p+1}) \e^p. \label{eq: Ito part}
\end{align}
Note that $C_4$ and $C_5$ are finite due to the existence of an exponential moment (\ref{eq: exponential moment}). 
Since the vector fields $\ffF$ and $(D \ffF) \ffF$ are globally Lipschitz continuous, 
we apply the estimates in Lemma 3.1 of \cite{KPP95}, which yields a constant $C_6 = C_6(p)>0$, 
then apply once again (\ref{eq: Gv estimate}) 
\begin{align}
\fI_t &:=  \Big|\sum_{0< s\lqq t} \Phi^{\ffF \Delta_s Z}(u^\e_{s-}, v^\e_{s-}) 
-\Phi^{\ffF \Delta_s Z}(u_{s-}, v_{s-}) \nonumber \\
&\qquad\qquad  -(u^{\e}_{s-} -u_{s-}, v^\e_{s-}- v_{s-}) 
- (\ffF(u^\e_{s-}, v^\e_{s-})-\ffF(u_{s-}, v_{s-})) \Delta_s Z\Big|^p\nonumber\\[3mm]
& \qquad \lqq  \Big(C_6 \sum_{0< s\lqq t} |(u^\e_{s-}-u_{s-}, v^\e_{s-} - v_{s-})| e^{C_6 \|\Delta_s Z\|}\|\Delta_s Z\|^2 \Big)^{p}\nonumber\\[2mm]
& \qquad \lqq (C_6)^p \Big(\sum_{0< s\lqq t} (|u^\e_{s-}-u_{s-}| +|v^\e_{s-} - v_{s-}|)e^{C_6 \|\Delta_s Z\|} \|\Delta_s Z\|^2\Big)^{p}\nonumber\\
& \qquad \lqq (2 C_6)^p \Big[\Big(\sum_{0< s\lqq t} |u^\e_{s-}-u_{s-}|e^{C_6 \|\Delta_s Z\|} \|\Delta_s Z\|^2\Big)^{p}
 +\Big(C_1\e t\sum_{0< s\lqq t} e^{C_6 \|\Delta_s Z\|} \|\Delta_s Z\|^2\Big)^{p}\Big].
\label{eq: Marcus estimate}
\end{align}
We go over to the representation with the random Poisson measure. 
For the first summand we obtain 
\begin{align}
&\sum_{0< s\lqq t} |u^\e_{s-}-u_{s-}|e^{C_6 \|\Delta_s Z\|} \|\Delta_s Z\|^2 \nonumber\\
&= \int_0^t \int_{\|y\|\lqq 1} |u^\e_{s-}-u_{s-}| e^{C_6 \|y\|} \|y\|^2 \ti N(dsdy) 
+\int_0^t \int_{\|y\|\lqq1} |u^\e_{s-}-u_{s-}| e^{C_6 \|y\|} \|y\|^2 ~\nu(dy)~ds \nonumber\\
&\qquad +\int_0^t \int_{\|y\|> 1} |u^\e_{s-}-u_{s-}| e^{C_6 \|y\|} \|y\|^2 N(dsdy)\nonumber\\
&= \int_0^t \int_{\RR^r} |u^\e_{s-}-u_{s-}| e^{C_6 \|y\|} \|y\|^2 \ti N(dsdy) 
+\int_0^t \int_{\RR^r} |u^\e_{s-}-u_{s-}| e^{C_6 \|y\|} \|y\|^2 ~\nu(dy)~ds \label{eq: Marcus estimate 1st term}
\end{align}
and 
\begin{align}
\sum_{0< s\lqq t} e^{C_6 \|\Delta_s Z\|} \|\Delta_s Z\|^2 
&= \int_0^t \int_{\|y\|\lqq 1} e^{C_6 \|y\|} \|y\|^2 \ti N(dsdy) 
+\int_0^t \int_{\|y\|\lqq1} e^{C_6 \|y\|} \|y\|^2 ~\nu(dy)~ds \nonumber\\
&\qquad +\int_0^t \int_{\|y\|> 1} e^{C_6 \|y\|} \|y\|^2 N(dsdy)\nonumber\\
&= \int_0^t \int_{\RR^r} e^{C_6 \|y\|} \|y\|^2 \ti N(dsdy) 
+\int_0^t \int_{\RR^r} e^{C_6 \|y\|} \|y\|^2 ~\nu(dy)~ds.\label{eq: Marcus estimate 2nd term}
\end{align}
Hence 
\begin{align}
\EE[\fI_t] 
&\lqq (2C_6)^p \bigg\{\EE\Big[\big|\int_0^t \int_{\RR^r} |u^\e_{s-}-u_{s-}| e^{C_6 \|y\|} \|y\|^2 \ti N(dsdy)\big|^p\Big] &\nonumber\\
& +\EE\Big[\big|\int_0^t \int_{\RR^r} |u^\e_{s-}-u_{s-}| e^{C_6 \|y\|} \|y\|^2 ~\nu(dy)~ds\big|^p\Big]\bigg\}\nonumber\\
& + (2 C_1 C_6 \e t)^p \bigg\{\EE\Big[\big|\int_0^t \int_{\RR^r} e^{C_6 \|y\|} \|y\|^2 \ti N(dsdy)\big|^p\Big] 
+ \big|\int_0^t \int_{\RR^r} e^{C_6 \|y\|} \|y\|^2 ~\nu(dy)~ds\big|^p \bigg\}&\nonumber\\
& =: J_1 + J_2 + J_3 + J_4. & \label{eq: major estimate}
\end{align}
Kunita's first theorem \cite{Ap09} provides a constant $C_7 = C_7(p)$ for the estimate of the $p$-th moment 
for integrals with respect to compensated random Poisson measure in $J_1$ and $J_4$. 
\begin{align}
J_1 &\lqq (2C_6 C_7)^p  \bigg\{ \EE\Big[\int_0^t \int_{\RR^r} |u^\e_{s-}-u_{s-}|^p e^{C_6 p \|y\|} \|y\|^{2p} \nu(dy)ds\Big] \nonumber\\
&\qquad + \EE\Big[\Big(\int_0^t \int_{\RR^r} |u^\e_{s-}-u_{s-}|^2 e^{C_6 2 \|y\|} \|y\|^{4} \nu(dy)ds\Big)^\frac{p}{2}\Big]\bigg\} \nonumber\\
&\lqq (2C_6 C_7)^p \bigg\{ \int_{\RR^r} e^{C_6 p \|y\|} \|y\|^{2p} \nu(dy)\;\EE\Big[ \int_0^t |u^\e_{s-}-u_{s-}|^p ds\Big] \nonumber\\
&\qquad + \Big(\int_{\RR^r} e^{2C_6 \|y\|} \|y\|^{4} \nu(dy)\Big)^\frac{p}{2}\; \EE\Big[\Big(\int_0^t|u^\e_{s-}-u_{s-}|^2 ds\Big)^\frac{p}{2} \Big]\bigg\} \nonumber\\
&\lqq (2C_6 C_7)^p(C_8+C_9)(T^{\frac{p}{2}-1}+1) \int_0^T \EE\left[ \sup_{s\in [0, t]} |u^\e_{s-}-u_{s-}|^p  \right] dt. \label{eq: J1}
\end{align}
Jensen's inequality estimates the remaining summands. We get 
\begin{align}
J_2 &\lqq (2C_6)^p  \int_{\RR^r} e^{C_6 p \|y\|} \|y\|^{2p} \nu(dy)\;T^{p-1}\int_0^T \EE\left[ \sup_{s\in [0, t]} |u^\e_{s-}-u_{s-}|^p  \right] dt\nonumber\\
&= (2C_6 C_{10})^pT^{p-1}\int_0^T \EE\left[ \sup_{s\in [0, t]} |u^\e_{s-}-u_{s-}|^p  \right] dt\label{eq: J2}\\
J_3 &\lqq (2C_1 C_6 C_7\e T)^p \bigg\{T^{\frac{p}{2}} \big|\int_{\RR^r} e^{2 C_6\|y\|} \|y\|^4 \nu(dy) \big|^\frac{p}{2}
+ T\int_{\RR^r} e^{p C_6 \|y\|} \|y\|^p \nu(dy)\bigg\}\nonumber\\
&\lqq (2C_1 C_6 C_7)^p (C_{11} + C_{10}) \e^p \bigg\{T^{\frac{3p}{2}}+ T^{p+1} \bigg\} \label{eq: J3} \\
J_4 &\lqq (2C_1 C_6 C_7)^p C_{12} \e^p T^{2p} \label{eq: J4} 
\end{align}
Taking the supremum and expectation in inequality (\ref{eq: pivot}) we combine 
(\ref{eq: Ito part}), (\ref{eq: Marcus estimate 1st term}) and (\ref{eq: Marcus estimate 2nd term}) 
with (\ref{eq: J1}), (\ref{eq: J2}), (\ref{eq: J3}) and (\ref{eq: J4}). 
Further we use $p+1\lqq \frac{3p}{2}\lqq 2p$ and $0 \lqq \frac{p}{2}-1\lqq p-1$  
and obtain positive constants $C_{13}$ and $C_{14}$ 
\begin{align*}
& \EE\left[\sup_{t\in [0, T]}|u_t^{\e} - u_t|^p\right] \\
& \lqq C_{13} (T^{2p} + T^{p+1})  \e^p + C_{14}\Big( T^{p-1} + 1\Big) 
\int_0^T \EE\left[\sup_{s\in [0, t]} |u_s^\e-u_s|^p \right] dt\\
& =: a_\e(T) + b(T) \int_0^T \EE\left[\sup_{s\in [0, t]} |u_s^\e-u_s|^p \right] dt.
\end{align*}
A standard integral version of Gronwall's inequality, as stated for instance in Lemma D.2 in \cite{SY02}, 
yields that 
\begin{align*}
\EE\left[\sup_{t\in [0, T]}|u_t^{\e} - u_t|^p\right] &\lqq a_\e(T)(1 + b(T) T \exp\left(b(T) T\right)) \\
&\lqq C_{13} T^{p+1} (1+ T^{p-1}) \e^p\Big[1 + C_{14} T (1+ T^{p-1})  \exp\left(C_{14} T (1+ T^{p-1})\right)\Big]\\[2mm]
&\lqq C_{13} T^{p} (1+ T)^p \e^p \exp\left(C_{15}T^{p})\right).
\end{align*}
Hence 
\begin{align}
\left(\EE\left[\sup_{t\in [0, T]}|u_t^{\e} - u_t|^p\right]\right)^{\frac{1}{p}} 
&\lqq C_{13}\, \e\, T(1+T) \exp\left(C_{15}\;T^p\right).\label{eq: final u estimate}
\end{align}
Eventually Minkowski's inequality and the estimates (\ref{eq: ungleichung 1}), 
(\ref{eq: Gv estimate}) and (\ref{eq: final u estimate}) yield the desired result 
\begin{align*}
&\left(\EE\left[\sup_{s\lqq T\wedge \tau^\e}|\Psi(X^\e_t(x_0) ) - \Psi(X_t(x_0))|^p\right]\right)^{\frac{1}{p}}\\
&\qquad \lqq C_0 \left(\EE\left[\sup_{s\lqq T\wedge \tau^\e}|u_s^\e - u_s|^p\right]\right)^{\frac{1}{p}} 
+C_0 \left(\EE\left[\sup_{s\lqq T\wedge \tau^\e}|v_s^\e - v_s|^p\right]\right)^{\frac{1}{p}}\\
&\qquad \lqq  C_0 C_{13}\, \e\, T(1+T) \exp\left(C_{15}\; T^p\right) + C_0 C_1 T\e \\[3mm]
&\qquad \lqq  C_{16}\, \e\, T  \exp\left(C_{17}\; T^p\right).
\end{align*}
This finishes the proof. 

\end{proof}

If $Z$ has a continuous component, the solution of equation (\ref{eq: SDE}) also contains a 
continuous Stratonovich component, see \cite{KPP95}. 
Combining the proof above with the proof of Lemma 2.1 in \cite{GR13} we conclude the following. 

\begin{cor}\label{cors: preliminary} Let $Z$ be a L\'evy process with characteristic triplet $(b, \nu, A)$ in $\RR^r$ 
for a drift vector $b\in \RR^r$ and the covariance matrix $A$ and $\nu$ as in Proposition \ref{lem: preliminary}. 
Then estimate 
(\ref{eq: compact bounded}) of Proposition 
\ref{lem: preliminary} holds true with an appropriate 
choice of the constants $k_1$ and $k_2$. 
\end{cor}

\begin{proof} 
The estimates of $\EE\left[\sup_{t\in [0, T]}|u_t^{\e} - u_t|^p\right]$ in both cases 
-continuous and pure jumps- just before applying Gronwall's lemma 
yields polynomial estimates in $T$ and $\e$ of the same degree. 
Hence Gronwall's inequality guarantees the same estimates modulo constants. 
\end{proof}

\section{Averaging functions on the leaves}

We recall that for a fixed $x_0 \in M$ and $\e>0$, 
 $\tau^\e $ denotes the exit time of $X^\e_{\cdot}(x_0)$ 
from the open neighbourhood $U\subset M$, 
which is diffeomorphic to $L_{x_0} \times V$. 

\begin{prp}\label{prp: componente vertical}
Given $\Psi: M\ra \RR$ differentiable and $Q^\Psi: V \ra \RR$ its average on the 
leaves given by formula (\ref{def: average}). 
For $t\gqq 0$ we denote by  
\[
\delta^\Psi(\e, t) := \int_{0}^{t \wedge \e\tau^{\e}} 
\Psi(X^\e_{\frac{r}{\e}}(x_0)) - Q^\Psi(\Pi(X^\e_{\frac{r}{\e}}(x_0))) dr.
\]
Then $\delta^\Psi(\e, t)$ tends to zero, when $t$ or $\e$ tend to zero. 
Moreover, if $Q^\Psi$ is $\alpha$-H\"older continuous with $\alpha>0$, then for 
any $\lambda <\alpha$, $p\gqq 2$ and any $\beta\in (0, \frac{1}{2})$ we have the following estimate 
\begin{align*}
\left(\EE\left[\sup_{s \lqq t} |\delta^\Psi(\e,s)|^p \right]\right)^\frac{1}{p} 
\lqq t \left[ \e^{\la} h(t, \e) +  \ \eta \left( t |\ln \epsilon |^{\frac{2\beta}{p}} \right) \right],
\end{align*}
where $h(t,\e)$ is continuous in $(t, \e)$ and tends to zero when $t$ or $\e$ do so. 
\end{prp}

\begin{proof} \textbf{(First part.)} 
For $\e$ sufficiently small and $t\gqq 0$ we define the partition 
\[
t_0 = 0 < t_1^\e < \dots < t_{N^\e}^\e \lqq \frac{t}{\e} \wedge \tau^{\e}
\]
as long as $X^\e$ has not left $U$ with the following step size 
\begin{equation*}
\Delta_\e := t |\ln\e|^{2\frac{\beta}{p}} 
\end{equation*}
by 
\[
t_n^\e := n \Delta_\e \qquad \mbox{ for }\qquad 0 \lqq n \lqq N^\e
\qquad \mbox{ where }\qquad N^\e = \lfloor (\e |\ln \e|^{2\frac{\beta}{p}})^{-1}\rfloor. 
\]
We now represent the first summand of $\delta^\Psi$ by 
\begin{align*}
\int_{0}^{t \wedge \e\tau^{\e}} \Psi(X^\e_{\frac{r}{\e}}(x_0)) dr &= \e 
\int_{0}^{\frac{t}{\e} \wedge \tau^{\e}} \Psi(X^\e_{r}(x_0)) dr\\
& = \e \sum_{n=0}^{N^{\e}-1} \int_{t_n}^{t_{n+1}} \Psi(X^\e_{r}(x_0)) + \e 
\int_{t_n}^{\frac{t}{\e} \wedge \tau^{\e}}\Psi(X^\e_{r}(x_0)) dr .
\end{align*}
We lighten notation and omit for convenience in the sequel all super and subscript $\e$ and $\Psi$ 
as well as the initial value $x_0$. 
We denote by $\theta$ the canonical shift operator on the canonical probability space $\Omega = D(\RR, M)$ 
of c\`adl\`ag functions. Let $F_t(\cdot, \omega)$ the stochastic flow of the original unperturbed system in $M$. 
The triangle inequality yields 
\begin{equation}\label{eq: delta decomposition}
|\delta^\Psi(\e, t)| \lqq |A_1(t, \e)| + |A_2(t, \e)| +|A_3(t, \e)| +|A_4(t, \e)|,  
\end{equation}
where 
\begin{align*}
A_1(t, \e) &:= \e \sum_{n=0}^{N-1} \int_{t_n}^{t_{n+1}} [\Psi(X^\e_{r}) -\Psi(F_{r-t_n}(X^\e_{t_n}, \theta_{t_n}(\omega)))] ~dr \\[2mm]
A_2(t, \e) &:= \e \sum_{n=0}^{N-1} \int_{t_n}^{t_{n+1}} 
[\Psi(F_{r-t_n}(X^\e_{t_n}, \theta_{t_n}(\omega))) - \Delta Q(\Pi(X^\e_{t_n}))] 
~dr \\[2mm]
A_3(t, \e) &:=\sum_{n=0}^{N-1}\e \Delta Q(\Pi(X^\e_{t_n})) - \int_{0}^{t \wedge 
\e\tau^{\e}}  Q(\Pi(X^\e_{\frac{r}{\e}})) ~dr \\[2mm]
A_4(t, \e) &:=  \e \int_{t_n}^{\frac{t}{\e} \wedge \tau^{\e}}\Psi(X^\e_{r}(x_0)) 
dr.
\end{align*}
The following four lemmas estimate the preceding terms. This being done the proof is finished. 

\begin{lem}
For any $\gamma \in (0,1)$ there exists a function $h_1 = h_1(\gamma)$ such that 
\begin{align*}
\left(\EE\left[\sup_{s\lqq t} |A_1(s, \e)|^p \right]\right)^{\frac{1}{p}} \lqq k_1 t \e^{\gamma} h_1(t, \e),
\end{align*}
where $h_1$ is continuous in $\e$ and $t$ and tends to zero when $\e$ and $t$ do so.  
\end{lem}
\begin{proof} The proof is identical to Lemma 3.2 in  \cite{GR13}, since Proposition \ref{lem: preliminary} 
provides the same asymptotic bounds as Lemma 2.1 in \cite{GR13}, which enters here. 
Furthermore, only the Markov property of the solutions of equation (\ref{eq: SDE perturbed}) is exploited. 
\end{proof}

\begin{lem} Let $\eta(t)$ be the rate of time convergence as defined in 
expression (\ref{def: function eta}). 
For the process $A_2$ in inequality (\ref{eq: delta decomposition}) we 
have:
\begin{align*}
\left(\EE\left[\sup_{s\lqq t} |A_2(s, \e)|^p \right]\right)^{\frac{1}{p}} \lqq 
  t \ \eta \left( t |\ln \epsilon |^{-\frac{2\beta}{p}} \right). 
\end{align*}
\end{lem}

\begin{proof} 

We have
\begin{eqnarray*}
 \left(\EE\left[\sup_{s\lqq t} |A_2(s, \e)|^p 
\right]\right)^{\frac{1}{p}}    & \lqq & 
\epsilon\  \sum^{N-1}_{n=0}
 \left[\mathbb{E}  \left|
\int^{t_{n+1}}_{t_n}{\Psi(F_{r-t_n}(X^{\epsilon}_{t_n},\theta_{t_n }
(\omega)))dr}- \Delta  
Q (\Pi (X^{\epsilon}_{ t_n}))   \right|^p\right]^{\frac{1}{p}} \\
&& \\
& = & \epsilon\ \Delta  \sum^{N-1}_{n=0}
 \left[\mathbb{E}  \left|
\frac{1}{\Delta }
\int^{t_{n+1}}_{t_n}{\Psi (F_{r-t_n}(X^{\epsilon}_{t_n},\theta_{t_n }
(\omega)))dr}- 
Q (\Pi (X^{\epsilon}_{ t_n}))   \right|^p\right]^{\frac{1}{p}}.
\end{eqnarray*}
For all $n=0, \ldots, N-1$, by the ergodic theorem, the two 
terms inside the
modulus converges to each other when $\Delta$ goes to infinity with rate of 
convergence bounded by  
$\eta (\Delta )$. Hence, for small $\epsilon$
we have
\begin{eqnarray*}
 \left(\EE\left[\sup_{s\lqq t} |A_2(s, \e)|^p \right]\right)^{\frac{1}{p}}& \lqq 
&  \epsilon\ N \, \Delta \, 
\eta(\Delta) \\
& \lqq&  \epsilon \left[ \epsilon^{-1} |\ln
\epsilon|^{-\frac{2\beta}{p}} \right] \  t 
|\ln \epsilon |^{\frac{2\beta}{p}} \eta \left( t |\ln \epsilon |^{\frac{2 
\beta}{p}} \right) \\
& =& t \ \eta \left( t |\ln \epsilon |^{\frac{2\beta}{p}} \right).
\end{eqnarray*}

\end{proof}

\begin{lem} 
Assume that $Q^\Psi$ is $\alpha$-H\"older continuous with $\alpha>0$. 
Then the process $A_3$ in inequality (\ref{eq: delta decomposition}) satisfies
\begin{align*} 
\left(\EE\left[\sup_{s\lqq t} |A_3(s,\e)|^p \right]\right)^{\frac{1}{p}} \lqq 
K_2  t^{1+\alpha} \e^\alpha |\ln(\e)|^{\frac{2\alpha \beta}{p}}
\end{align*}
for a positive constant $K_2>0$. 
\end{lem}
\begin{proof} We lighten notation $Q = Q^\Psi$. We consider the interval $[0, t]$ 
with the partition $0 < \e t_1 < \dots < \e t_N \lqq t$
\begin{align*}
|A_3(t,\e)| &=  \Big| \sum_{n=0}^{N-1} \e \Delta Q(\Pi(X^\e_{t_n}))
- \int_{0}^{t \wedge \e \tau^{\e}}  Q(\Pi(X^\e_{\frac{r}{\e}})) ~dr \Big|\\
& \lqq \e \sum_{n=0}^{N-1} \Delta \sup_{\e t_n \lqq s < \e t_{n+1}}  
|Q(\Pi(X^\e_s)) - Q(\Pi(X^\e_{t_n}))|\\
& \lqq \e C_1 \Delta N \sup_{\e t_n \lqq s < \e t_{n+1}} |v^\e_s - 
v^\e_{t_n}|^\alpha\\
& \lqq \e C_2 \Delta N (\e h)^\alpha \\
& \lqq K_2 \e^{1+\alpha} t^{1+\alpha} |\ln(\e)|^{(1+\alpha) \frac{2\beta}{p}} 
\e^{-1} |\ln(\e)|^{- \frac{2\beta}{p}}\\
& = K_2  t^{1+\alpha} \e^{\alpha} |\ln(\e)|^{\frac{2\alpha \beta}{p}}.
\end{align*}
\end{proof}
\begin{lem} 
The process $A_4$ satisfies  
\begin{equation*}
\EE\Big[\sup_{s\lqq t} |A_4(s,\e)|^p\Big] \lqq K_3 t \e |\ln 
\e|^{\frac{2\beta}{p}},  
\end{equation*}
where $K_3 = \|\Psi\|_{\infty, U}$.
\end{lem}
The proof follows straightforward, see also the proof of Lemma 3.5 in 
\cite{GR13}.\\

\noindent \textbf{(Final step of Proposition \ref{prp: componente vertical})} 
Collecting the results of the previous lemmas yields with the help of Minkowski's inequality the desired result
\begin{align*}
\left(\EE\left[\sup_{s\lqq t} |\delta^\Psi(\e, s)|^p\right]\right)^{\frac{1}{p}} 
&\lqq t  \left(k_1  
\e^\gamma h_1(\e, t) +  \eta \left( t |\ln \epsilon |^{-\frac{2\beta}{p}} 
\right) 
+ k_2  t^{\alpha} \e^\alpha |\ln(\e)|^{\frac{2\alpha 
\beta}{p}}+ K_3 \, \e \ |\ln \e|^{\frac{2\beta}{p}}\right) \\
&=: t \left[ \e^{\la} h(t, \e) +  \ \eta \left( t |\ln \epsilon 
|^{\frac{2\beta}{p}} \right) \right],
\end{align*}
where $h(t,\e)$ tends to zero if $\e$ or $t$ does so, for all $\la<\alpha$. 
\end{proof}

\section{Proof of the main result}

The gradient of each $\Pi_i$ is orthogonal to the leaves. 
Hence by It\^o's formula for canonical Marcus integrals, 
see e.g. \cite{KPP95} Proposition 4.2, 
we obtain for $i=1, \dots, d$ that 
\begin{equation}
\Pi_i\big(X^\e_{\frac{t}{\e}\wedge \tau^{\e} }\big) 
= \int_0^{\frac{t \wedge \tau^{\e}}{\e}} d\Pi_i(\e K)(X^\e_r) dr
= \int_0^{t \wedge \tau^{\e}} d\Pi_i(K)(X^\e_\frac{r}{\e}) dr.
\end{equation}
  We may continue and change the variable 
\begin{align*}
|\Pi_i\big(X^\e_{\frac{t}{\e}\wedge \tau^{\e} }) - w_i \big(t\big)| 
& \lqq \int_0^{t\wedge \tau^{\e}} |Q^{d\Pi_i(K)}(X^\e_\frac{r}{\e}) -  
Q^{d\Pi_i(K)}(w(r))| dr + |\delta^{d\Pi_i}(t, \e)|\\
& \lqq C_1 \int_0^{t\wedge \tau^{\e}}  |\Pi_i(X^\e_\frac{r}{\e}) -  w_i(r)| dr + 
|\delta^{d\Pi_i}(t, \e)|\\
& \lqq C_2 \int_0^{t\wedge \tau^{\e}} |\Pi(X^\e_\frac{r}{\e}) -  w(r)| dr + 
\sum_{i=1}^N |\delta^{d\Pi_i}(t, \e)|.\\
\end{align*}
Since the right-hand side does not depend on $i$, 
we can sum over $i$ at the left-hand side 
and apply Gronwall's lemma. This yields 
\begin{align*}
|\Pi\big(X^\e_{\frac{t}{\e}\wedge \tau^{\e} }) - w \big(t\big)| \lqq e^{C_2 t } 
\sum_{i=1}^d |\delta^{d\Pi_i}(t, \e)|.
\end{align*}
An application of Proposition \ref{prp: componente vertical} finishes the proof of the first statement. 
For the second part we calculate with the help of Chebyshev's inequality 
\begin{align*}
\PP\big(\e \tau^{\e} < T_\gamma) & \lqq \PP\big(\sup_{s\lqq T_\gamma \wedge \e 
\tau^{\e}} 
| \Pi(X^\e_{\frac{t}{\e}\wedge \tau^{\e} }) - w(s)| >\gamma \big) \\
& \lqq \gamma^{-p} \EE\left[\sup_{s\lqq T_\gamma \wedge \e \tau^{\e}} | 
\Pi(X^\e_{\frac{t}{\e}\wedge \tau^{\e} }) - w(s)|^p\right] \\
& \lqq \gamma^{-p}  T_{\gamma}^p \left[ \e^{\la} h(T_{\gamma}, \e) +  \ \eta 
\left( T_{\gamma} |\ln 
\epsilon 
|^{-\frac{2\beta}{p}} \right) e^{CT_{\gamma}} \right]^p.
\end{align*}

\section{Appendix} 


\begin{lem}\label{lem: ergodicity Gamma process}
For any uniformly distributed random value in $[0, 2\pi)$ and $p\in \CC$ the Gamma process defined in (\ref{eq: Levy-Ito}) 
satisfies  
\begin{align*}
\EE[e^{ip e^{i Z_t}}] \ra \EE[e^{ip e^{iU}}] \qquad \mbox{ as } t \ra \infty.  
\end{align*}
\end{lem}

\begin{proof} The marginal densities of the Gamma process are well-known 
to be $p_t(x) = \frac{\theta^{t}}{\Gamma(t)} x^{t-1}e^{-\theta x} dx$ such that  
\begin{align*}
\EE[e^{i(ue^{iL_t})}] 
&= \int_\RR e^{iue^{x}} \frac{\theta^{t}}{\Gamma(t)} x^{t-1}e^{-\theta x} dx.
\end{align*}
Using the oddness of the imaginary part, the evenness of the real part and 
the $2\pi$-periodicity of $x\mapsto e^{iue^{x}}$ for any $u\in \CC$ we obtain 
\begin{align*}
\EE[e^{i(ue^{iL_t})}] 
&= 2 \sum_{n=1}^\infty  \int_{0}^{2\pi} \Re e^{iue^{y}} \frac{\theta^{t}}{\Gamma(t)} (2\pi n + y)^{t-1}e^{-\theta (2\pi n + y)} dy\\
&= 2 \int_{0}^{2\pi}  \Re e^{iue^{y}} \frac{\theta^{t}}{\Gamma(t)}  \bigg(\sum_{n=1}^\infty(2\pi n + y)^{t-1}e^{-\theta (2\pi n + y)}\bigg) dy.
\end{align*}
We continue for fixed $y$ with the density. It behaves as 
\begin{align*}
\sum_{n=1}^\infty(2\pi n + y)^{t-1}e^{-\theta (2\pi n + y)}
&\sim \int_{0}^\infty  (2\pi v + y)^{t-1}e^{-\theta (2\pi v + y)} dv\\
&= \frac{1}{2\pi} \int_{y}^\infty  z^{t-1}e^{-\theta z} dz\\
&= \frac{\Gamma(t)}{2\pi \theta^{t}} -\int_0^y \frac{z^{t-1}e^{-\theta z}}{2\pi} dz.
\end{align*}
Hence 
\begin{align*}
\EE[e^{i(ue^{iL_t})}] 
&= 2 \int_{0}^{2\pi}  \Re e^{iue^{y}} \frac{\theta^{t}}{\Gamma(t)} \Big(\frac{\Gamma(t)}{2\pi \theta^{t}} -\frac{y^{t}e^{-\theta y}}{2\pi}\Big)dy\\
&= 2 \int_{0}^{2\pi} \frac{\Re e^{iue^{y}}}{2\pi}dy - 2 \frac{\theta^{t}}{\Gamma(t)} \int_{0}^{2\pi}e^{iue^{y}} 
\int_0^y \frac{\si^{t}e^{-\theta \si}}{2\pi} d\si dy.
\end{align*}
For the remainder term we use large Stirling's formula of $\Gamma(t) \approx \sqrt{2\pi t} \Big(\frac{t}{e}\Big)^t$ for large $t$ 
and the Beppo-Levi theorem 
\begin{align*}
\frac{1}{\Gamma(t)} \int_{0}^{2\pi}\Re e^{iue^{y}} 
\theta^{t}\int_0^y \frac{\si^{t}e^{-\theta \si}}{2\pi} d\si dy 
&\lqq  \int_{0}^{2\pi}\Re e^{iue^{y}} 
\frac{(2\pi \theta)^{t}}{\Gamma(t)} dy  \ra 0, \qquad \mbox{ as } t \ra \infty. 
\end{align*}
This concludes the proof. 
\end{proof}

\paragraph{Calculations: } In the sequel we verify (\ref{def: function eta}) for the radial component 
using the elementary identity $\cos^2(z) =\Re\;\frac{1}{2}(e^{i2z}+1)$, $z\in \RR$. 
\begin{align*}
\mathbb{E} \frac{1}{t}\int_0^t d\Pi_r K (X_s(x_0) )\, ds
&= \mathbb{E} \frac{1}{t}\int_0^t r \cos^2(Z_s) ds \\
&= \frac{r}{t} \int_0^t \mathbb{E}\cos^2(Z_s) ds \\
&= \frac{r}{t} \int_0^t \frac{1}{2}\Re\; \mathbb{E}[\exp(i2Z_s)] ds + \frac{r}{t} \int_0^t \frac{1}{2} ds\\
&= \frac{r}{t} \int_0^t \frac{1}{2} \exp(- Cs) ds + \frac{r}{2}\stackrel{t\nearrow\infty}{\lra} \frac{r}{2}.
\end{align*}
This yields the convergence in $p=1$. For $p=2$ we continue 
\begin{align*}
&\EE\Big[\Big|\frac{1}{t}\int_0^t r \cos^2(Z_s) ds - \frac{r}{2}\Big|^2\Big]\\
&=\EE\Big[\Big(\frac{r}{t}\Big)^2 \Big(\int_0^t \cos^2(Z_s) ds\Big)^2 - \frac{r^2}{t}\int_0^t \cos^2(Z_s) ds + \frac{r^2}{4}\Big] \\
&= \Big(\frac{r}{t}\Big)^2 \EE\Big[\iint_{0}^t \cos^2(Z_s)\cos^2(Z_\si) dsd\si\Big] 
 - \frac{r^2}{t}\int_0^t \EE\Big[\cos^2(Z_s)\Big] ds + \frac{r^2}{4} \\
&= \Big(\frac{r}{t}\Big)^2 2\int_{0}^t\int_0^\si \EE\Big[\cos^2(Z_s)\cos^2(Z_\si)\Big] dsd\si
 + \frac{r^2}{t} \int_0^t \frac{1}{2} \exp(- Cs) ds - \frac{r^2}{4}.
\end{align*}
Using the elementary identity $\cos^2(x) \cos^2(y) = \cos(2(x+y)) + \cos(2(x-y)) +2\cos(2x) + 2\cos(2y) + 2$
we continue with the first term
\begin{align*}
&\Big(\frac{r}{t}\Big)^2 2\int_{0}^t\int_0^\si \EE\Big[\cos^2(Z_s)\cos^2(Z_\si)\Big] dsd\si\\
&= \Big(\frac{r}{t}\Big)^2 \frac{1}{4}
\int_{0}^t\int_0^\si \bigg(
\EE\Big[e^{i2 (Z_\si - Z_s)}\Big]  + \EE\Big[e^{i2 (Z_\si + Z_s)}\Big] + 2 \EE\Big[e^{i2 Z_\si}\Big] 
+ 2 \EE\Big[e^{i2 Z_s}\Big] +2\bigg) 
dsd\si\\
&= \Big(\frac{r}{t}\Big)^2 \frac{1}{4}
\int_{0}^t\int_0^\si \bigg(
\EE\Big[e^{i2 (Z_\si - Z_s)}\Big]  + \EE\Big[e^{i2 (Z_\si - Z_s)}\Big] \EE\Big[e^{i4 Z_s}\Big] + 2 \EE\Big[e^{i2 Z_\si}\Big] 
+ 2 \EE\Big[e^{i2 Z_s}\Big] +2\bigg) dsd\si\\
&= \Big(\frac{r}{t}\Big)^2 \frac{1}{4}
\int_{0}^t\int_0^\si \bigg(
e^{-C (\si-s)}  + e^{-C \si} e^{(C_4-C) s)} + 2 e^{-C \si}+ 2 e^{-C s} +2\bigg) 
dsd\si\\
&= \Big(\frac{r}{t}\Big)^2 \frac{1}{4}
\int_{0}^t \bigg(
e^{-C \si}\frac{1}{C}(e^{C\si}-1)  + e^{-C \si} \frac{1}{C}(e^{(C-C_4) \si}-1) + 2\si e^{-C \si}+ \frac{2}{C} (1-e^{-C \si}) +2\si \bigg) 
d\si\\
&= \Big(\frac{r}{t}\Big)^2 \frac{1}{4}
\int_{0}^t \bigg(
\frac{1}{C}(1-e^{-C\si})  +  \frac{1}{C}(e^{-C_4 \si}-e^{-C \si}) + 2\si e^{-C \si}+ \frac{2}{C} (1-e^{-C \si}) +2\si \bigg) 
d\si\\
&\lqq \Big(\frac{r}{t}\Big)^2 \frac{1}{4} (a + bt + t^2) \ra \frac{r^2}{4}.
\end{align*}
Combining the previous two results and 
and taking the square root, the rate of convergence is of order $\eta(t) = C/\sqrt{t}$ as $t\nearrow \infty$. 

\bigskip

\noindent {\bf Acknowledgements:} 
The first author would like to express his gratitude to Berlin Mathematical School (BMS), 
International Research Training Group (IRTG) 1740: Dynamical Phenomena in Complex Networks: 
Fundamentals and Applications and the probability group at Humboldt-Universit\"at zu Berlin 
and Universit\"at Potsdam for various infrastructural support. 
He would also like to thank Paulo Henrique Pereira da Costa for 
fruitful discussions. 

This article has been written while the
second author was in a sabbatical visit to Humboldt-Universit\"at zu Berlin. 
He would like to express his gratitude to Prof. Peter Imkeller and his research group for
the nice and friendly hospitality. Partially supported by FAPESP
11/50151-0, 12/03992-1 and CNPq 477861/2013-0.


\end{document}